\documentclass[a4paper]{article}

\usepackage[utf8]{inputenc}
\usepackage[T1]{fontenc}
\usepackage{textcomp}
\usepackage{amsmath, amssymb}
\usepackage{amsthm}
\usepackage{mathrsfs}
\usepackage{esint}

\usepackage{import}
\usepackage{xifthen}
\usepackage{pdfpages}
\usepackage{transparent}
\usepackage{hyperref}
\usepackage{subcaption}
\usepackage{caption}
\usepackage[most]{tcolorbox}

\newtheorem{problem}{Problem}[section]
\newtheorem{theorem}[problem]{Theorem}
\newtheorem{proposition}[problem]{Proposition}
\newtheorem{corollary}[problem]{Corollary}

\newtheorem{remark}[problem]{Remark}

\theoremstyle{definition} 
\newtheorem{definition}[problem]{Definition}
\begin{document}
\title{Generalized Bernstein Theorem for Stable Minimal Plateau Surfaces}
\author{Gaoming Wang}
\date{}
\maketitle

\abstract{In this paper, we consider a Generalized Bernstein Theorem for a type of generalized minimal surfaces, namely minimal Plateau surfaces. We show that if an orientable minimal Plateau surface is stable and has quadratic area growth in $\mathbb{R}^3 $, then it is flat.}

\section{Introduction}%
\label{sec:introduction}

%\noindent\textbf{Blowdown argument}

When we study minimal surfaces, we usually focus on smooth minimal surfaces. But in general, "non-smooth minimal surfaces" can be found in some natural phenomena. Indeed, two types of singular points are observed in soap films and they are documented in Plateau's work \cite{plateau1873laws}.
Hence, a natural problem raised by Bernstein and Maggi \cite{Bernstein2021} is,
\begin{problem}
	\label{prob_Bern}
	To what extent may the classical theory of minimal surfaces be generalized to "non-smooth minimal surfaces"?
\end{problem}

In \cite{Bernstein2021}, they defined those "non-smooth minimal surfaces" as \textit{minimal Plateau surfaces}.

To better illustrate our theorem here, let us define the minimal Plateau surfaces first.
We define three cones in $\mathbb{R}^3 $ by
\begin{align*}
	P:={} & \{ p=(x_1,x_2,x_3) \in \mathbb{R}^3 : x_3=0 \},\\
	Y:={} & \{ p=(r\cos \theta, r \sin \theta,x_3)\in \mathbb{R}^3 :\theta=0, \frac{2\pi}{3} \text{ or }\frac{4\pi}{3}, r\ge 0 \},\\
	T:={}& \{p= ap_i+bp_j:a,b\ge 0, 1\le i<j\le 4 \},
\end{align*}
where $p_i$ are defined as
\[
	p_1=(1,1,1),\quad p_2=(-1,-1,1),\quad p_3=(-1,1,-1),\quad p_4=(1,-1,-1).
\]

For the following definition, we fix an open set $U\subset \mathbb{R}^3$, $\alpha \in (0,1)$, and a relatively closed subset $\Sigma \subset U$, we use $\mathrm{reg}(\Sigma)$ to denote the set of all $p \in \Sigma$ such that $B_r(p)\cap \Sigma $ is a $C^{1,\alpha}$ surface for $r>0$ small enough and $B_r(p)\subset U$.
\begin{definition}
	\label{def_plateau}
	(\cite{Bernstein2021})
	We say a relatively closed subset $\Sigma \subset U$ is a \textit{Plateau surface} if the following two conditions hold,
\begin{itemize}
	\item For any $p \in \Sigma$, there is an $r>0$ and a $C^{1,\alpha}$ diffeomorphism $\phi:B_r(p)\rightarrow \mathbb{R}^3 $ such that $\phi(\Sigma\cap B_r(p))=\boldsymbol{C}\cap B_r(p)$ and $D\phi_p \in O(3)$ for some $\boldsymbol{C}=P,Y$ or $T$.
	\item Each connected component of $\mathrm{reg}(\Sigma)$ has constant mean curvature.
\end{itemize}
In addition, if each connected component of $\mathrm{reg}(\Sigma)$ has zero mean curvature, then we say $\Sigma$ is a \textit{minimal Plateau surface} in $U$.
%In particular, if $\Sigma$ only satisfies the first condition above, we say $\Sigma$ is a Plateau surface.
\end{definition}

We note the \textit{tangent cone} $T_p\Sigma:=\lim_{\rho\rightarrow 0^+} \frac{\Sigma-p}{\rho}$ for any $p \in \Sigma$ can only be isometric to $P,Y$ or $T$ if $\Sigma$ is a minimal Plateau surface.
In particular, we use the following notations to denote $Y$-type and $T$-type singular sets.
\begin{align*}
	%\mathrm{reg}(\Sigma):={} & \left\{ p \in \Sigma:T_p\Sigma \text{ is a plane} \right\},\\
	\Sigma_Y:={} & \left\{ p \in \Sigma:T_p \Sigma\text{ is isometric to $Y$} \right\},\\
	\Sigma_T:={} & \left\{ p \in \Sigma:T_p \Sigma\text{ is isometric to $T$} \right\},
	%\mathrm{sing}(\Sigma):={}&\Sigma_Y \cup \Sigma_T.
\end{align*}

%Then, we can change the second condition in Definition \ref{def_plateau} to the condition that the mean curvature of $\mathrm{reg}(\Sigma)$ vanishes everywhere on $\mathrm{reg}(\Sigma)$ by first varional formula.

For a minimal Plateau surface $\Sigma$ in $U$, if $\Sigma_T=\emptyset $ and $\Sigma_Y\neq \emptyset $, we call it \textit{a minimal triple junction surface}. (It was named as $Y$-surface in \cite{Bernstein2021}.)
Indeed, the author has studied the properties of minimal triple junction surfaces in his thesis \cite{thesis} and in \cite{Wang2022curvature} as a special case of multiple junction surfaces.

Now we can state our main theorem.

\begin{theorem}
	
	Suppose $\Sigma$ is a minimal Plateau surface in $\mathbb{R}^3$. We assume $\Sigma$ is orientable, complete, stable and has at most quadratic area growth.
	Then $\mathrm{reg}(\Sigma)$ is flat.
	\label{thm_main_Plat}
\end{theorem}

Here, the related concepts are defined in Section \ref{sec:preliminary}.
As a corollary, we can give a more precise description of minimal Plateau surfaces satisfying Theorem \ref{thm_main_Plat}.

\begin{corollary}\label{cor_main_plat}
	If $\Sigma$ is a minimal Plateau surface satisfying the conditions in Theorem \ref{thm_main_Plat}, then each component of $\mathrm{reg}(\Sigma)$ is an open subset of some plane, $\Sigma_Y$ is a union of disjoint line segments, rays, and straight lines.

	In particular, if $\Sigma_T=\emptyset $, then we can write $\Sigma=N\times \mathbb{R} $ after some right motions in $\mathbb{R}^3 $ where $N$ is an embedded stationary network in $\mathbb{R}^2 $. 

	If $\Sigma_T\neq \emptyset $, there are only two possibilities. The first one is $\Sigma_T$ contains only one point and $\Sigma$ is isometric to $T$.
	The second one is $\Sigma_T$ contains two points and we can glue two $T$-type sets to get $\Sigma$ as shown in Figure \ref{fig:flat-tplanes}.
\end{corollary}

\begin{figure}[ht]
    \centering
	\begingroup
	\fontsize{7pt}{12pt}
	\def\svgwidth{0.6\columnwidth}
	%% Creator: Inkscape 1.2 (dc2aeda, 2022-05-15), www.inkscape.org
%% PDF/EPS/PS + LaTeX output extension by Johan Engelen, 2010
%% Accompanies image file 'flat-tplanes.pdf' (pdf, eps, ps)
%%
%% To include the image in your LaTeX document, write
%%   \input{<filename>.pdf_tex}
%%  instead of
%%   \includegraphics{<filename>.pdf}
%% To scale the image, write
%%   \def\svgwidth{<desired width>}
%%   \input{<filename>.pdf_tex}
%%  instead of
%%   \includegraphics[width=<desired width>]{<filename>.pdf}
%%
%% Images with a different path to the parent latex file can
%% be accessed with the `import' package (which may need to be
%% installed) using
%%   \usepackage{import}
%% in the preamble, and then including the image with
%%   \import{<path to file>}{<filename>.pdf_tex}
%% Alternatively, one can specify
%%   \graphicspath{{<path to file>/}}
%% 
%% For more information, please see info/svg-inkscape on CTAN:
%%   http://tug.ctan.org/tex-archive/info/svg-inkscape
%%
\begingroup%
  \makeatletter%
  \providecommand\color[2][]{%
    \errmessage{(Inkscape) Color is used for the text in Inkscape, but the package 'color.sty' is not loaded}%
    \renewcommand\color[2][]{}%
  }%
  \providecommand\transparent[1]{%
    \errmessage{(Inkscape) Transparency is used (non-zero) for the text in Inkscape, but the package 'transparent.sty' is not loaded}%
    \renewcommand\transparent[1]{}%
  }%
  \providecommand\rotatebox[2]{#2}%
  \newcommand*\fsize{\dimexpr\f@size pt\relax}%
  \newcommand*\lineheight[1]{\fontsize{\fsize}{#1\fsize}\selectfont}%
  \ifx\svgwidth\undefined%
    \setlength{\unitlength}{680.31496063bp}%
    \ifx\svgscale\undefined%
      \relax%
    \else%
      \setlength{\unitlength}{\unitlength * \real{\svgscale}}%
    \fi%
  \else%
    \setlength{\unitlength}{\svgwidth}%
  \fi%
  \global\let\svgwidth\undefined%
  \global\let\svgscale\undefined%
  \makeatother%
  \begin{picture}(1,0.5)%
    \lineheight{1}%
    \setlength\tabcolsep{0pt}%
    \put(0,0){\includegraphics[width=\unitlength,page=1]{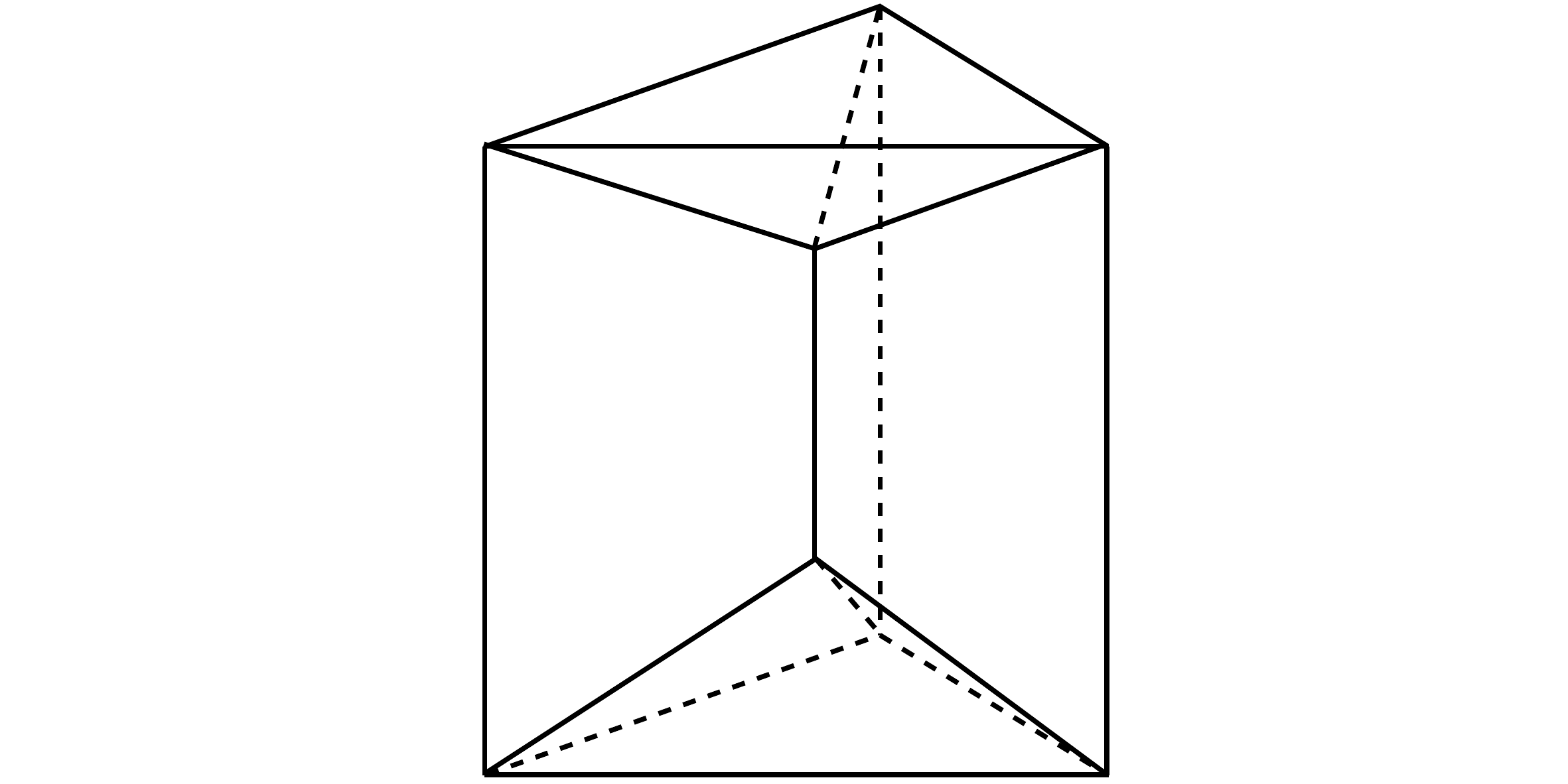}}%
    \put(0.47911615,0.31903352){\color[rgb]{0,0,0}\makebox(0,0)[lt]{\lineheight{0}\smash{\begin{tabular}[t]{l}$p_1$\end{tabular}}}}%
    \put(0.48059948,0.14845035){\color[rgb]{0,0,0}\makebox(0,0)[lt]{\lineheight{0}\smash{\begin{tabular}[t]{l}$p_2$\end{tabular}}}}%
  \end{picture}%
\endgroup%

	\endgroup

    \caption{A flat $\Sigma$ with $\Sigma_T=\left\{ p_1,p_2 \right\}$}
    \label{fig:flat-tplanes}
\end{figure}

Theorem \ref{thm_main_Plat} can be viewed as an extension of our previous work \cite{Wang2022curvature} in two aspects.
\begin{itemize}
	\item We allow $T$-type singularities.
	\item We can remove the restriction on $\Gamma$ appearing in \cite[Theorem 1]{Wang2022curvature}. See Appendix \ref{sec:appendix} for details.
\end{itemize}

Before the formal definition of Plateau surfaces in \cite{Bernstein2021}, Plateau proposed well-known laws (known as Plateau's laws) to describe the structure of soap films.
The definition of Plateau surfaces is just the mathematical way to describe those soap bubbles and soap films obeying Plateau's laws.

In 1976, Almgren \cite{almgren1976LipMinimizing} started studying closed sets which minimize Hausdorff measures with respect to local Lipschitz deformations, proving that they are smooth minimal surfaces out of a closed set of null area. In the same year the work of Taylor \cite{taylor1976structure} appeared with the sharp regularity in $\mathbb{R}^3$, and the singular set $\Sigma_Y$ and $\Sigma_T$ appeared. Her result justified Plateau's laws mathematically.
%In particular, Taylor observed that there are two types of singularities appearing. One is a $Y$-type singularity and another is a $T$-type singularity.
%These singularities just correspond to the singular sets $\Sigma_Y$ and $\Sigma_T$. 
Using Taylor's result, Choe \cite{Choe1989fundamentalDomain} could obtain the regularity of the fundamental domains with the least boundary area.
In particular, singularities of $Y$-type and $T$-type appeared there.

Besides, $Y$-type and $T$-type singularities are also quite natural in the sense of stationary varifolds.
If we know a stationary integral 2-varifold $V$ is sufficent close to $Y$ in a open ball $B_1(0)\subset \mathbb{R}^3 $ and we have an area bound like $\|V\|(B_1)<2$ and a density bound $\Theta(\|V\|,0)\ge \frac{3}{2}$, then by Simon's celebrated paper \cite{Simon1993}, we know $\mathrm{spt}\|V\|$ is indeed a $C^{1,\alpha}$ perturbation of $Y$ in a smaller ball $B_{\frac{1}{2}}(0)$.
This result has been extended to the case of polyhedral cones by Colombo, Edelen, and Spolaor \cite{Edelen2022PolyhedralCone}.
Roughly speaking, we can get if a stationary integral 2-varifold $V$ is sufficient close to $T$ in $B_1(0)$ and has a suitable area and density bounds, then we can conclude $\mathrm{spt}\|V\|$ is a $C^{1,\alpha}$ perturbation of $T$ in $B_{\frac{1}{2}}(0)$.

Inspired by those minimizing properties and stationary properties for minimal Plateau surfaces, we expect there are any other results which can be extended to minimal Plateau surfaces to give some positive answers to Problem \ref{prob_Bern}.
Of course, Bernstein and Maggi \cite{Bernstein2021} showed they can get the rigidity of the $Y$-shaped catenoid.

On the other hand, there are indeed some other properties held for minimal triple junction surfaces compared with the usual minimal surfaces.
For example, Mese and Yamada \cite{Mese2006} showed they could solve the singular version of the Plateau problem using mapping methods in some cases.

These results motivate us to seek whether stable minimal Plateau surfaces have similar results like curvature estimate and generalized Bernstein theorem for stable minimal surfaces.
In my PhD thesis \cite{thesis}, we studied the triple junction surfaces intrinsically and showed we can solve some particular type of elliptic partial differential equations on it.
After that, we find several concepts like the Morse index are also well-defined on triple junction surfaces.

%Another reason we care about the curvature estimates and generalized Bernstein Theorem is, that it is closely related to the min-max construction of singular minimal surfaces in an arbitrary ambient manifold.
%The min-max method developed by Almgren and Pitts is a powerful tool to find the minimal hypersurfaces in any Riemannian manifolds, see for instance \cite{Pitts1976min-max, Smith1983min_max, DeLellis2013embedded_minimal}.
%One of the successful applications of the min-max theory is due to Marques and Neves. They solved several longstanding conjectures including the Willmore conjecture \cite{marques2021morse}.
%Later on, there is much more research on the min-max construction of hypersurfaces in various settings like free boundary minimal hypersurfaces \cite{li2016existence}, constant mean curvature hypersurface \cite{ZhouZhu2019CMC}, capillary minimal surfaces \cite{Li2021Capillary}, and so on.
%These constructions all need some types of curvature estimates and generalized Bernstein Theorem for hypersurfaces with stability conditions like \cite{Schoen_1975, Schoen1983, GuangLiZhou2020curvatureEstimate} and \cite[Appendix C]{Li2021Capillary}.
%Hence, if we want to construct nontrivial minimal triple junction surfaces or even minimal Plateau surfaces, it is quite natural to obtain a generalized Bernstein Theorem for minimal Plateau surfaces at first.

For usual stable minimal hypersurfaces, we know that the generalized Bernstein Theorem is equivalent to some curvature estimates via a standard blow-up argument.
But for stable minimal Plateau surfaces, things get complicated.
This is because even if we assume the second fundamental form of each surface is uniformly bounded, we cannot make sure any sequence of minimal Plateau surfaces has a convergence subsequence.
For example, if we only consider the triple junction surfaces shown in Figure \ref{fig:degenerate}, we find some of $\Sigma^i$ may become degenerate after taking limits.
\begin{figure}[ht]
    \centering
	\begingroup
	\fontsize{7pt}{12pt}
	\def\svgwidth{0.5\columnwidth}
	%% Creator: Inkscape 1.2-beta (1b65182c, 2022-04-05), www.inkscape.org
%% PDF/EPS/PS + LaTeX output extension by Johan Engelen, 2010
%% Accompanies image file 'degenerate.pdf' (pdf, eps, ps)
%%
%% To include the image in your LaTeX document, write
%%   \input{<filename>.pdf_tex}
%%  instead of
%%   \includegraphics{<filename>.pdf}
%% To scale the image, write
%%   \def\svgwidth{<desired width>}
%%   \input{<filename>.pdf_tex}
%%  instead of
%%   \includegraphics[width=<desired width>]{<filename>.pdf}
%%
%% Images with a different path to the parent latex file can
%% be accessed with the `import' package (which may need to be
%% installed) using
%%   \usepackage{import}
%% in the preamble, and then including the image with
%%   \import{<path to file>}{<filename>.pdf_tex}
%% Alternatively, one can specify
%%   \graphicspath{{<path to file>/}}
%% 
%% For more information, please see info/svg-inkscape on CTAN:
%%   http://tug.ctan.org/tex-archive/info/svg-inkscape
%%
\begingroup%
  \makeatletter%
  \providecommand\color[2][]{%
    \errmessage{(Inkscape) Color is used for the text in Inkscape, but the package 'color.sty' is not loaded}%
    \renewcommand\color[2][]{}%
  }%
  \providecommand\transparent[1]{%
    \errmessage{(Inkscape) Transparency is used (non-zero) for the text in Inkscape, but the package 'transparent.sty' is not loaded}%
    \renewcommand\transparent[1]{}%
  }%
  \providecommand\rotatebox[2]{#2}%
  \newcommand*\fsize{\dimexpr\f@size pt\relax}%
  \newcommand*\lineheight[1]{\fontsize{\fsize}{#1\fsize}\selectfont}%
  \ifx\svgwidth\undefined%
    \setlength{\unitlength}{680.31496063bp}%
    \ifx\svgscale\undefined%
      \relax%
    \else%
      \setlength{\unitlength}{\unitlength * \real{\svgscale}}%
    \fi%
  \else%
    \setlength{\unitlength}{\svgwidth}%
  \fi%
  \global\let\svgwidth\undefined%
  \global\let\svgscale\undefined%
  \makeatother%
  \begin{picture}(1,0.5)%
    \lineheight{1}%
    \setlength\tabcolsep{0pt}%
    \put(0,0){\includegraphics[width=\unitlength,page=1]{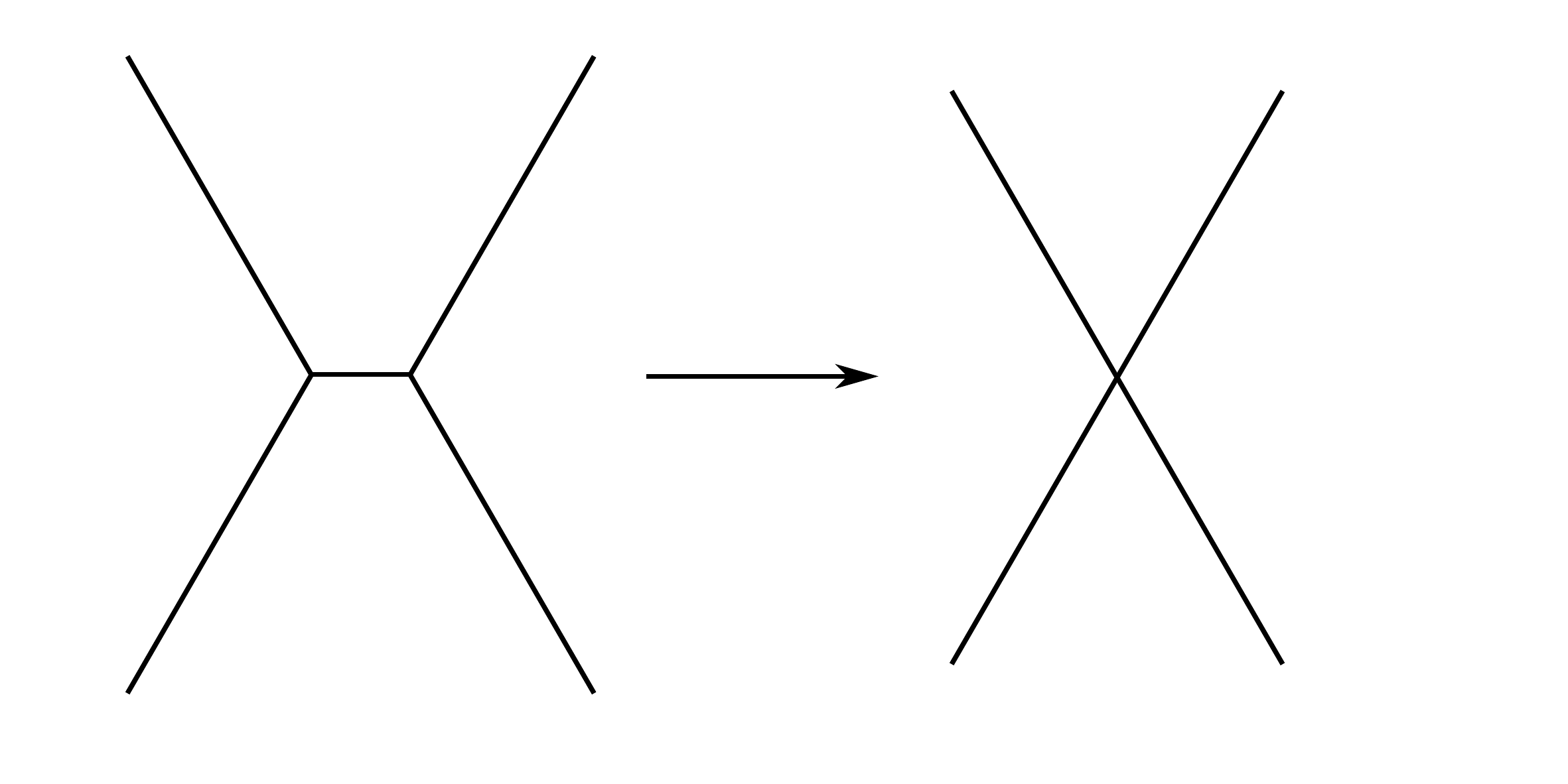}}%
    \put(0.2087946,0.26989956){\color[rgb]{0,0,0}\makebox(0,0)[lt]{\lineheight{0}\smash{\begin{tabular}[t]{l}$\Sigma^i$\end{tabular}}}}%
    \put(0,0){\includegraphics[width=\unitlength,page=2]{degenerate.pdf}}%
    \put(0.76660879,0.28374563){\color[rgb]{0,0,0}\makebox(0,0)[lt]{\lineheight{0}\smash{\begin{tabular}[t]{l}$\Sigma^i$, degenerate\end{tabular}}}}%
  \end{picture}%
\endgroup%

	\endgroup

    \caption{$\Sigma^i$ is degenerate after taking limit}
    \label{fig:degenerate}
\end{figure}

Therefore, we cannot directly get a pointwise curvature estimate for stable multiple junction surfaces.
At least, we need a very careful description of the limit when some of $\Sigma^i$ become degenerate.
We hope the works like \cite{Schoen1983} and \cite{Simon1993} may be useful in the later studying of curvature estimates for stable triple junction surfaces.

The organization of the paper is as follows.
In Section \ref{sec:preliminary}, we fix several notations and review some basic properties of minimal Plateau surfaces.
In Section \ref{sec:variations_of_plateau_surfaces}, we define some function spaces arising from the variations and give the second variation formula for minimal Plateau surfaces.
In Section \ref{sec:main_theoreem}, we can prove our main theorem by using logarithmic cutoff functions.

\section{Notations}%
\label{sec:preliminary}

It is convenient to introduce a class of generalized surfaces in $U$ to describe the variation of minimal Plateau surfaces.

\begin{definition}
	\label{def_general_Plateau}
	Given an open set $U\subset \mathbb{R}^3 $, we say $\Sigma$ is \textit{a generalized surface} in $U$ and write $\Sigma \in \mathcal{S}(U)$ if $\Sigma$ is (relatively) closed in $U$ and for some $\alpha \in (0,1)$ and any $p \in U$, there is an $r>0$ and a $C^{1,\alpha}$ diffeomorphism $\phi:B_r(p)\rightarrow \mathbb{R}^3$ such that $\phi(\Sigma\cap B_r(p))=\boldsymbol{C}\cap B_r(p)$ for some $\boldsymbol{C}=P,Y$ or $T$.

	We say $\Sigma \subset \mathcal{S}(U)$ is \textit{complete} if $U=\mathbb{R}^3$.

	%If $U=\mathbb{R}^3 $, then we say $\Sigma$ is a \textit{complete Plateau surface} in $\mathbb{R}^3$.
\end{definition}

Note that the Plateau surfaces are in the class $\mathcal{S}(U)$.

%To describe the orientability of Plateau surfaces, let us introduce the notion of cell structures defined in \cite{Bernstein2021}.

\begin{definition}
	\label{def_cell_stru}
	We say that $\Sigma \in \mathcal{S}(U)$ is \textit{orientable} in $U$ if each connected component $V\subset U\backslash \Sigma$ satisfies the following condition.
	\begin{itemize}
		\item For each $p \in \Sigma$, there is a $\rho>0$ with $B_\rho(p)\subset U$ and for all $0<r<\rho$, we have $B_r(p)\cap V$ is connected.
	\end{itemize}
	%We say $\Sigma$ is \textit{orientable} if and only if $\Sigma$ defines a cell structure in $U$.
\end{definition}

Note that Definition \ref{def_cell_stru} is closely related to the notion of "$\Sigma$ having a cell structure in $U$", which was introduced in \cite[Section 1.4]{Bernstein2021} to define a concept of orientability on non-smooth surfaces.
With respect to that definition, the notion of orientable surface we propose is different in that we are not assuming $U\backslash\Sigma$ to have finitely many connected components.

\begin{definition}
	\label{def_quadratic}
	We say a complete generalized surface $\Sigma \in \mathcal{S}(U)$ has at most \textit{quadratic area growth} if there exists a positive constant $C$ such that for any $r>0$ and $p \in \Sigma$, we have
	\[
		\left|\Sigma \cap B_r(p)\right|\le Cr^2.
	\]
\end{definition}
Here, we will use $\left|\Sigma\right|$ to denote the area of $\Sigma$. Precisely, $\left|\Sigma\right|:=\mathcal{H}^2(\Sigma)$ where $\mathcal{H}^n$ is the Hausdorff measure.

In the following part of this paper, we always assume $\Sigma \in \mathcal{S}(U)$ is orientable.
%\begin{remark}
%	Note that the definitions for Plateau surfaces and cell structures are slightly different from those that appeared in \cite{Bernstein2021}.
%\end{remark}
% Let $\Sigma$ be a Plateau surface.
Note that we can define $\Sigma_Y,\Sigma_T$ for $\Sigma \in \mathcal{S}(U)$.
Precisely, we write,
\begin{align*}
	%\mathrm{reg}(\Sigma):={}& \left\{ p \in \Sigma:T_p \text{ is a plane} \right\},\\
	\Sigma_Y:={} & \left\{ p \in \Sigma:T_p \Sigma\text{ is homeomorphic to $Y$} \right\},\\
	\Sigma_T:={} & \left\{ p \in \Sigma:T_p \Sigma\text{ is homeomorphic to $T$} \right\},\\
	\mathcal{A}:={} & \left\{ \Lambda \subset \mathrm{reg}\Sigma:\Lambda \text{ is a connected component of }\mathrm{reg}\Sigma \right\},\\
	\mathcal{L}:={} & \left\{ L \subset \Sigma_Y: L  \text{ is a connected component of }\Sigma_Y\right\}.
\end{align*}
%Note that $\Sigma_T$ is a discrete set by Definition \ref{def_plateau}.

For each $\Lambda \in \mathcal{A}$, we know $\Lambda \subset U$ is a $C^{1,\alpha}$ surface with piecewise $C^{1,\alpha}$ boundaries.
We write $\partial \Lambda=\overline{\Lambda}\cap U$.
We say $\nu$ is a unit normal vector field on $\Sigma$ if we can write $\nu=(\nu_\Lambda)_{\Lambda \in \mathcal{A}}$ and each $\nu_\Lambda$ is a unit normal vector field on $\Lambda$ for each $\Lambda \in \mathcal{A}$.
Similarly, we can define the following geometric quantities and functional spaces on $\Sigma$ as follows.
\begin{itemize}
	\item $A:=(A_\Lambda)_{\Lambda \in \mathcal{A}}$ denotes the second fundamental form on $\Sigma$ where $A_\Lambda$ is the second fundamental form on $\Lambda$.
	\item $H:=(H_\Lambda)_{\Lambda \in \mathcal{A}}$ denote the mean curvature of $\Sigma$ with respect to the normal vector field $\nu$.
		We choose the sign of $H_\Lambda$ such that the unit sphere has mean curvature 2 with respect to the unit normal vector field pointing inside of the sphere.
	\item $C^{k,\alpha}(\Sigma):=\{ f=(f_{\Lambda})_{\Lambda\in \mathcal{A}}:f_\Lambda \in C^{k,\alpha}(\overline{\Lambda}) \}$ denotes the $C^{k,\alpha}$ function space on $\Sigma$.
		% TODO
	\item $W^{k,p}(\Sigma):=\{ f=(f_{\Lambda})_{\Lambda\in \mathcal{A}}:f_\Lambda \in W^{k,p}(\overline{\Lambda}) \}$ denotes the Sobolev spaces on $\Sigma$. 
		%\tcbox{TODO: Special case.}
	\item $C^{k,\alpha}(\Sigma,\mathbb{R}^3 ):=\{ X=(X_p)_{p \in \Sigma}:X_p \in T_p \mathbb{R}^3 $ and $X\lfloor(\overline{\Lambda}) $ is a $C^{k,\alpha}$ differentiable vector field on $ \overline{\Lambda} $ for each $\Lambda \in \mathcal{A}\}$ denotes the set of all $C^{k,\alpha}$ vector fields on $\Sigma$.
		%In general, we do not require $X \in C^{k}(\Sigma,T \mathbb{R}^3 )$ to be a restriction of $C^k$ tangential vector field on $\mathbb{R}^3 $.
\end{itemize}

For any $L \in \mathcal{L}$, we write $\left\{ \Lambda_L^1,\Lambda_L^2,\Lambda_L^3 \right\}$ as the collection of $\Lambda \in \mathcal{A}$ with $L \subset \partial \Sigma^i_L$ for $i=1,2,3$.
Note that for each $\Lambda \in \mathcal{A}$, $\Lambda$ is a $C^{1,\alpha}$ surface with piecewise $C^{1,\alpha}$ boundaries and $\Lambda^i_L$ is the surface that has $L$ as one of the boundary curves.
For any $f \in C^{k}(\Sigma)$ or $f \in W^{k,p}(\Sigma)$, we write $f_L^i=f_{\Lambda^i_L}$ for any $L \in \mathcal{L}$ and $1\le i\le 3$.
We also write $\nu^i_L=\nu _{\Lambda^i_L}$.
%At last, we 

Similarly, for any $Q \in \Sigma_T$, we write $\left\{ L_Q^1,L_Q^2,L_Q^3,L_Q^4 \right\}$ as the collection of $L \in \mathcal{L}$ having $Q$ as its endpoints.

\begin{remark}
	Note that $\Lambda_L^i$ cannot be the same with $\Lambda_L^j$ for $i\neq j$ if we assume $\Sigma$ is orientable.
	Similarly, $L^i_Q, L^j_Q$ are all disjoint curves if $i\neq j$ for any $Q \in \Sigma_T$.
\end{remark}

\section{Variations of minimal Plateau surfaces}%
\label{sec:variations_of_plateau_surfaces}

\subsection{Function spaces araising from variations}%
\label{sub:function_spaces_araising_from_variations}

Before talking about the first variation formula for $\Sigma \in \mathcal{S}(U)$, let us consider the definitions of variations and related function spaces.

\begin{definition}
	\label{def_variation}
	Given an open set $U\subset \mathbb{R}^3 $, for any $\Sigma \subset \mathcal{S}(U)$, we say $\Sigma_t=\varphi_t(\Sigma)$, $t \in (-\varepsilon,\varepsilon)$ is a $C^{1,\alpha}$ variation of $\Sigma$ if $\varphi_t$ is a family of $C^{1,\alpha}$ diffeomorphisms supported in $U$ such that $\varphi_0$ is an identity.

	We say $\varphi_t$ has compact support if $\varphi_t$ is an identity outside a compact domain $W\subset U$ for any $t$.
\end{definition}

We write $V(x):=\frac{d}{dt}|_{t=0}\varphi_t(x)$ as the variational vector field on $\Sigma$ for the variation $\Sigma_t$ and we can find $V \in C^{1,\alpha}(\Sigma,\mathbb{R}^3 )$.
Note that if $V \in C^{1,\alpha}(\Sigma,\mathbb{R}^3 )$, then the function $f$ defined by
\[
	f:=(f_\Lambda)_{\Lambda \in \mathcal{A}}, f_\Lambda=V \cdot \nu_\Lambda, \text{ for any }\Lambda \in \mathcal{A},
\]
is a function in $C^{1,\alpha}(\Sigma)$.
Here, $V\cdot \nu_\Lambda$ means the inner product in $\mathbb{R}^3$.
We write $f=V\cdot \nu$ for short.
Note that not all $f \in C^{1,\alpha}(\Sigma)$
 can be written as $f=V\cdot \nu$ for some $V \in C^{1,\alpha}(\Sigma)$.
We define a subspace of $C^{1,\alpha}(\Sigma)$ as follows.
\begin{definition}
	\label{def_compatible}
	We say a function $f \in C^{1,\alpha}(\Sigma)$ satisfies the \textit{compatible condition} if $f$ can be written as $f=V\cdot \nu$ for some $V \in C^{1,\alpha}(\Sigma,\mathbb{R}^3)$.
	We use $C^{1,\alpha}_2(\Sigma)$ to denote the set of all $f \in C^{1,\alpha}(\Sigma)$ satisfying the compatible condition.
\end{definition}

\begin{remark}
	Another important function space is defined as
	\[
		C_1^{1,\alpha}(\Sigma):=\left\{ f \in C^{1,\alpha}(\Sigma):f^i_L=f^j_L, 1\le i,j\le 3 \text{ for each }L \in \mathcal{L}\right\}.
	\]
	This function space is usually related to the restriction of the functions to $\Sigma$.
	In particular, it plays an important role when we want to construct triple junction surfaces.
	These spaces are generalized in \cite{thesis}.
\end{remark}

%\tcbox{Another way to describe orientability.}

Indeed, we have different definitions of compatible conditions.
Before that, we need to define a sign function to describe the orientation of $\Sigma$.
Note that the sign functions have been defined in my thesis \cite{thesis} for triple junction hypersurfaces.
We extend its definition to generalized surface $\Sigma \in \mathcal{S}(U)$ here.

For each $L \in \mathcal{L}$, we choose an oritentation on $L$. 
Then we define the sign function $\mathrm{sign}=(\mathrm{sign}_L)_{L \in \mathcal{L}}$ with $\mathrm{sign}_L=(\mathrm{sign}^1_L,\mathrm{sign}_L^2,\mathrm{sign}_L^3)\in C^{\infty}(L,\mathbb{R}^3)$ and we define $\mathrm{sign}^i_L$ to be a constant function on $L$ such that $\mathrm{sign}^i_L\equiv 1$ if the orientation on $L$ agrees with the induced orientation from $\Lambda_L^i$ and its orientation $\nu^i_L$.
We define $\mathrm{sign}^i_L \equiv -1$ otherwise.
Note that by our definition of sign function and fact $\sum_{i =1}^{3}\tau^i_L=0$ on $L$, we have $\sum_{i =1}^{3}\mathrm{sign}^i_L(p)\nu_L^i(p)=0$ for any $p \in L$.

%\begin{align*}
%	\mathrm{sign}_L={}&(\mathrm{sign}^1_L,\mathrm{sign}_L^2,\mathrm{sign}_L^3)\quad \text{ and }\\
%	\mathrm{sign}^i_L(p)={}&
%	\begin{cases}
%	1, & \text{the orientation induced by }\nu^i_\Lambda \\
%	, & 
%	\end{cases}
%\end{align*}

\begin{proposition}
	\label{prop_equi_comp}
	The function $f \in C^{1,\alpha}(\Sigma)$ satisfies the compatible condition if and only if $f$ satisfies
	\[
		\sum_{i =1}^{3}\mathrm{sign}^i_L(p)f^i_{L}(p)=0, \text{ for any }p \in L, \text{ and any }L \in \mathcal{L}.
	\]
\end{proposition}

\begin{proof}
	If $f=V\cdot \nu$, we can easily find $\sum_{i =1}^{3}\mathrm{sign}^i_L(p)f^i_L(p)=0$ since we know that $\nu$ satisfies $\sum_{i =1}^{3}\mathrm{sign}^i_L \nu^i_L=0$.

	On the other hand, for any $f \in C^{1,\alpha}_2(\Sigma)$, we need to construct $V \in C^{1,\alpha}(\Sigma,\mathbb{R}^3)$ such that $f=V\cdot \nu$.

	We need to use the partition of unity to finish the construction.
	For any $p \in \Sigma$, we let $B_r(p)$ be the ball in Definition \ref{def_plateau} such that $\Sigma\cap B_r(p)$ is $C^{1,\alpha}$ diffeomorphic to $\boldsymbol{C}\cap B_r(p)$ by the diffeomorphism $\phi$ with $D\phi_p \in O(3)$ for $\boldsymbol{C}=P,Y$ or $T$.
	In the later construction, we will extend $\nu_\Lambda$ to a $C^{1,\alpha}$ vector field and $f_\Lambda$ a $C^{1,\alpha}$ function near $\Lambda$ if necessary.

	%For any $\Lambda \in \mathcal{A}$, $p \in \Lambda$, we can choose a small neighborhood $B_r(p)\subset U$ such that $B_r(p)\cap \Sigma=B_r(p)\cap \Lambda$.
	Firstly, if $p \in \Lambda$ for some $\Lambda \in \mathcal{A}$, we can choose $V=f_\Lambda \nu_\Lambda$. Secondly, if $p \in L$ for some $L \in \mathcal{L}$, we can choose $V=\frac{2}{3} \sum_{i =1}^{3}f^i_L \nu^i_L$. This time we need to verify $f^i_L=V\cdot\nu^i_L$.
	Note that by the definition of sign functions, we have $\nu^i_L \cdot \nu^j_L=-\frac{1}{2}\mathrm{sign}^i_L \mathrm{sign}^j_L$ if $i\neq j$.
	Hence, using the compatible condition $\sum_{i =1}^{3}f^i_L \mathrm{sign}^i_L=0$, we have,
	\begin{align*}
		V\cdot \nu^j_L={}& \frac{2}{3}\sum_{i =1}^{3}f_L^i \nu_L^i \cdot \nu_L^j=\frac{2}{3}f^j_L-\frac{1}{3}\sum_{i \neq j}^{}f^i_L \mathrm{sign}^i_L\mathrm{sign}^j_L\\
		={}&f^j_L - \frac{1}{3}\mathrm{sign}^j_L\sum_{i =1}^{3} \mathrm{sign}^i_L
		f^i_L=f^j_L.
	\end{align*}

	At last, we suppose $p \in \Sigma_T$. We use $\Lambda^{ij}_p$ to denote the $\Lambda \in \mathcal{A}$ such that $L^i_p\cup L^j_p \subset \partial \Lambda$.
	We write $f^{ij}_p=f_{\Lambda^{ij}_p}$ and $\nu^{ij}_p=\nu_{\Lambda^{ij}_p}$.
	In particular, we write the sign function $\mathrm{sign}^{ij}_p:=\mathrm{sign}^i_{L^j_p}$, which is the sign function defined on $L^j_p$ induced by surface $\Lambda^{ij}_p$ and its norm $\nu^{ij}_p$.
	Hence, the compatible condition implies $\sum_{1\le i\neq j \le 4}^{}f^{ij}_p \mathrm{sign}^{ij}_p=0$ along $L_p^j$ for each $j$.
	(Note that we know $f^{ij}_p=f^{ji}_p,\nu^{ij}_p=\nu^{jl}_p$ but $\mathrm{sign}^{ij}_p$ are different from $\mathrm{sign}^{ji}_p$ for $i\neq j$.)
	%\tcbox{TODO: Special case.}
	Now we define
	\[
		V=\frac{1}{2}\sum_{1\le i< j\le 4 }^{}f^{ij}_p\nu^{ij}_p.
	\]

	We note the inner product between $\nu^{ij}_p, \nu^{kl}_{p}$ is given by
	\[
		\nu^{ij}_p\cdot \nu^{kl}_{p}=
		\begin{cases}
		1, & (i,j)=(k,l) \text{ or }(l,k), \\
		-\frac{1}{2}\mathrm{sign}^{ji}_p \mathrm{sign}^{li}_p, & i=k,\\
		0, & (i,j,k,l) \text{ is a permutation of }(1,2,3,4),
		\end{cases}
	\]
	using the definition of sign functions.
	On the other hand, we can use the compatible condition to get
	\begin{align*}
		V\cdot \nu^{kl}_p={} & \frac{1}{4}\sum_{1\le i\neq j\le 4 }^{}f^{ij}_p\nu^{ij}_p \cdot \nu^{kl}_p\\
		={}&\frac{1}{2}f^{kl}_p-\frac{1}{4}\sum_{ \substack{1\le i\le 4\\i\neq k,l} }^{} f^{ki}_p\mathrm{sign}^{ik}_p \mathrm{sign}^{lk}_p-\frac{1}{4}\sum_{ \substack{1\le i\le 4\\i\neq k,l} }^{} f^{li}_p\mathrm{sign}^{il}_p \mathrm{sign}^{kl}_p\\
		={} & \frac{1}{2}f^{kl}_p+\frac{1}{4}f^{lk}_p \mathrm{sign}^{lk}_p \mathrm{sign}^{kl}_p+\frac{1}{4}f^{lk}_p \mathrm{sign}^{kl}_p \mathrm{sign}^{kl}_p=f^{kl}_p.
	\end{align*}

	Now, we can use the partition of unity to glue them into a vector field $V$ near $\Sigma$ with $f=V\cdot \nu$.
\end{proof}

\begin{remark}
	Inspired by Proposition \ref{prop_equi_comp}, we can define the Sobolev space
	\[
		W^{k,p}_{2}(\Sigma):=\left\{ 
		f \in W^{k,p}(\Sigma):\sum_{i =1}^{3}\mathrm{sign}^i_L f^i_L|_{L}=0 \text{ a.e. on }L,\forall L \in \mathcal{L}\right\},
	\]
	where the restriction should be understood in a trace sense.
\end{remark}

At last, we would like to mention the locally constant functions on $\Sigma$ defined below.
\begin{definition}
	\label{def_const_fun}
	We say $f \in C_2^{1,\alpha}(\Sigma)$ is a locally constant function if $f_\Lambda$ is a constant for each $\Lambda \in \mathcal{A}$.
\end{definition}

Note that in general, the dimension of the space of all constant functions could be larger than 1.
In particular, we can construct some locally constant functions based on orientations.
Let $\Omega \subset U\backslash \Sigma$ be a connected component of $U\backslash \Sigma$, and we define $f=f^{\Omega}=(f^\Omega_\Lambda)_{\Lambda \in \mathcal{A}}$ by
\[
	f^\Omega_\Lambda:=
	\begin{cases}
	1, & \Lambda \subset \overline{\Omega} \text{ and }\nu_\Lambda \text{ pointing inward with respect to $\Omega$.}\\
	-1, & \Lambda \subset \overline{\Omega} \text{ and }\nu_\Lambda \text{ pointing outward with respect to $\Omega$.}\\
	0, & \Lambda \not\subset \overline{\Omega}.
	\end{cases}
\]

Using the definition of sign functions, it is easy to verify $f^\Omega_\Lambda \in C_2^{1,\alpha}(\Sigma)$.

%\begin{remark}
%	The construction of constant functions here is also used in \cite[Section 5]{Wang2021network} to give a lower bound of index of stationary networks.
%\end{remark}

\begin{remark}
	Using locally constant functions, we can describe the orientability in another way. That is, $\Sigma$ is orientable if and only if there exists a locally constant function $f \in C^{1,\alpha}_2(\Sigma)$ which is nonvanishing everywhere.
	Of course, we need to extend our definitions of those function spaces to non-orientable generalized surfaces.

	Here, we provide an example to explain why we call it orientable if we can find a nonvanishing locally constant function on $\Sigma$.
	Suppose $\Sigma_1,\Sigma_2$ are two squares $[0,1]\times [0,1]$ in $\mathbb{R}^2 $ and we glue their boundaries to get $\Sigma$.
	We write $L_1=\left\{ 0 \right\}\times [0,1]$ and $L_2=\left\{ 1 \right\}\times [0,1]$.
	If we glue $\Sigma_1,\Sigma_2$ showed in Figure \ref{fig:mobius-strip} to get a M\"obuis strip $\Sigma$, then the compatible condition on $\Sigma$ is given by $f_1=f_2$ on $L_1$ and $f_1=-f_2$ on $L_2$.
	It is impossible to find a nonvanishing locally constant function on $\Sigma$.
\begin{figure}[ht]
    \centering
	\begingroup
	\fontsize{7pt}{12pt}
	\def\svgwidth{0.6\columnwidth}
	%% Creator: Inkscape 1.2 (dc2aeda, 2022-05-15), www.inkscape.org
%% PDF/EPS/PS + LaTeX output extension by Johan Engelen, 2010
%% Accompanies image file 'mobius-stripe.pdf' (pdf, eps, ps)
%%
%% To include the image in your LaTeX document, write
%%   \input{<filename>.pdf_tex}
%%  instead of
%%   \includegraphics{<filename>.pdf}
%% To scale the image, write
%%   \def\svgwidth{<desired width>}
%%   \input{<filename>.pdf_tex}
%%  instead of
%%   \includegraphics[width=<desired width>]{<filename>.pdf}
%%
%% Images with a different path to the parent latex file can
%% be accessed with the `import' package (which may need to be
%% installed) using
%%   \usepackage{import}
%% in the preamble, and then including the image with
%%   \import{<path to file>}{<filename>.pdf_tex}
%% Alternatively, one can specify
%%   \graphicspath{{<path to file>/}}
%% 
%% For more information, please see info/svg-inkscape on CTAN:
%%   http://tug.ctan.org/tex-archive/info/svg-inkscape
%%
\begingroup%
  \makeatletter%
  \providecommand\color[2][]{%
    \errmessage{(Inkscape) Color is used for the text in Inkscape, but the package 'color.sty' is not loaded}%
    \renewcommand\color[2][]{}%
  }%
  \providecommand\transparent[1]{%
    \errmessage{(Inkscape) Transparency is used (non-zero) for the text in Inkscape, but the package 'transparent.sty' is not loaded}%
    \renewcommand\transparent[1]{}%
  }%
  \providecommand\rotatebox[2]{#2}%
  \newcommand*\fsize{\dimexpr\f@size pt\relax}%
  \newcommand*\lineheight[1]{\fontsize{\fsize}{#1\fsize}\selectfont}%
  \ifx\svgwidth\undefined%
    \setlength{\unitlength}{680.31496063bp}%
    \ifx\svgscale\undefined%
      \relax%
    \else%
      \setlength{\unitlength}{\unitlength * \real{\svgscale}}%
    \fi%
  \else%
    \setlength{\unitlength}{\svgwidth}%
  \fi%
  \global\let\svgwidth\undefined%
  \global\let\svgscale\undefined%
  \makeatother%
  \begin{picture}(1,0.5)%
    \lineheight{1}%
    \setlength\tabcolsep{0pt}%
    \put(0,0){\includegraphics[width=\unitlength,page=1]{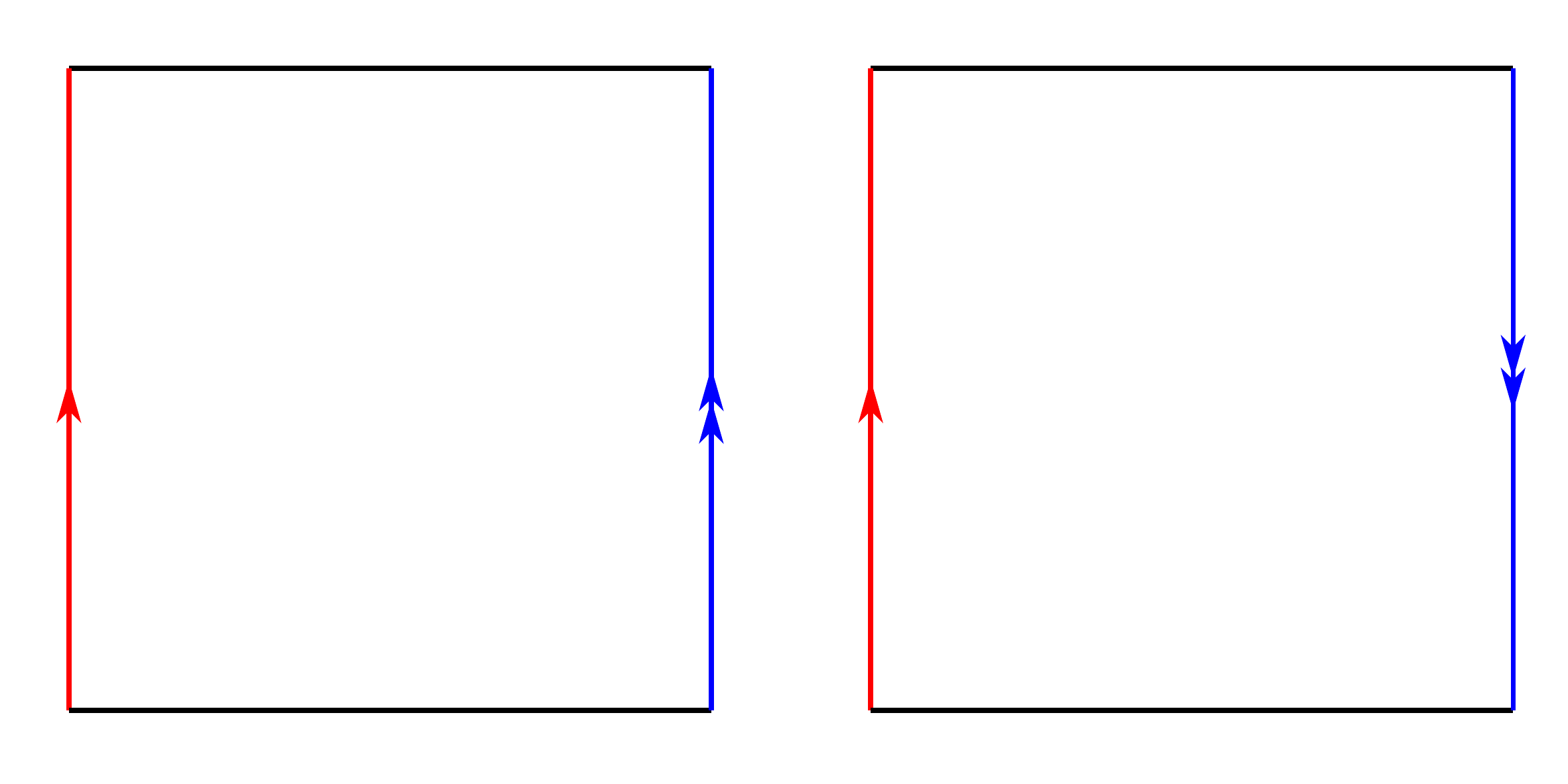}}%
    \put(0.15017982,0.26614855){\color[rgb]{0,0,0}\makebox(0,0)[lt]{\lineheight{0}\smash{\begin{tabular}[t]{l}$\Sigma_1$\end{tabular}}}}%
    \put(0.65435499,0.26614855){\color[rgb]{0,0,0}\makebox(0,0)[lt]{\lineheight{0}\smash{\begin{tabular}[t]{l}$\Sigma_2$\end{tabular}}}}%
    \put(0.0493448,0.11060512){\color[rgb]{0,0,0}\makebox(0,0)[lt]{\lineheight{0}\smash{\begin{tabular}[t]{l}$L_1$\end{tabular}}}}%
    \put(0.55995621,0.11060512){\color[rgb]{0,0,0}\makebox(0,0)[lt]{\lineheight{0}\smash{\begin{tabular}[t]{l}$L_1$\end{tabular}}}}%
    \put(0.40548553,0.11060512){\color[rgb]{0,0,0}\makebox(0,0)[lt]{\lineheight{0}\smash{\begin{tabular}[t]{l}$L_2$\end{tabular}}}}%
    \put(0.9214605,0.11060512){\color[rgb]{0,0,0}\makebox(0,0)[lt]{\lineheight{0}\smash{\begin{tabular}[t]{l}$L_2$\end{tabular}}}}%
  \end{picture}%
\endgroup%

	\endgroup

    \caption{M\"obius strip}
    \label{fig:mobius-strip}
\end{figure}
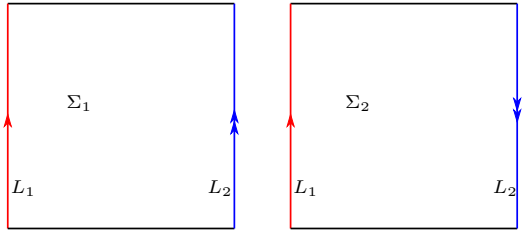
\end{remark}

\subsection{Second variation formula}%
\label{sub:second_variation_formula}

Given an open set $U \subset \mathbb{R}^3$, we suppose $\Sigma \in \mathcal{S}(U)$ and $\Sigma_t=\varphi_t(\Sigma)$ is a $C^{1,\alpha}$ variation of $\Sigma$ with compact support in $W\subset \subset U$.
Let $V\subset C^{1,\alpha}(\Sigma,\mathbb{R}^3 )$ be the variational vector field associated with variation $\Sigma_t$.
Then we can compute the variation of the area to get
\[
	\left.\frac{d}{dt}\right|_{t=0}
		\left|\Sigma_t\cap W\right|=-\sum_{\Lambda \in \mathcal{A} }^{}\int_{ \Lambda} H_\Lambda \phi d\mu_\Lambda+\sum_{L \in \mathcal{L} }^{}\sum_{i =1}^{3}\int_{ L} \tau^i_L\cdot V d\mu_L.
\]

Hence, we find $\Sigma$ is minimal if and only if $\Sigma$ is a critical point of its area functional.

Now let us recall the second variational formula that appeared in \cite{Wang2022curvature}.
Note that the computation is valid for minimal Plateau surfaces and hence we do not repeat that proof here.
Interested readers may refer to \cite{Wang2022curvature} for more details.

\begin{theorem}
	%(cf. \cite[Theorem 4]{Wang2022curvature})
	Let $\Sigma$ be a minimal Plateau surface in $U$ and $\Sigma_t$ the variation of $\Sigma$ with compact support in $W\subset \subset U$.
	Then, we have
	\[
		\left.\frac{d^2}{dt^2}\right|_{t=0}
			\!\!\!\!
			\left|\Sigma_t\cap W\right|\!=\!
			\sum_{\Lambda \in \mathcal{A} }^{}\!\int_{ \Lambda} \!\left|\nabla_\Lambda \phi_\Lambda\right|^2-\left|A_\Lambda\right|^2\!\phi_\Lambda^2d\mu_\Lambda-\sum_{L\in \mathcal{L}}^{}\!
			\sum_{i =1}^{3}\!\int_{ L}\! ( \phi^i_L )^2\boldsymbol{H}_L\cdot \tau^i_Ld\mu_L.
		%\sum_{i =1}^{q}
		%\int_{ \Sigma^i} \left( \left|\nabla _{\Sigma^i}\phi^i\right|^2-\left|A _{\Sigma^i}\right|^2\left(\phi^i\right)^2\right)\theta^i d\Sigma^i-\int_{ \Gamma} (\phi^i)^2 \boldsymbol{H}_\Gamma \cdot \tau^i \theta^i d\Gamma.
	\]
	where $\phi=V\cdot \nu$, $V$ is the variational vector field associated with variation $\Sigma_t$, and $\boldsymbol{H}_L$ is the geodesic curvature vector of $L$ in $\mathbb{R}^3$.
\end{theorem}

We say a minimal Plateau surface $\Sigma$ is \textit{stable} if for every variation $\Sigma_t$ of $\Sigma$ with support in $W\subset \subset U$, we have $\left. \frac{d ^2}{d t^2}\right|_{t=0}\left|\Sigma_t\cap W\right|\ge 0$.

Note that by an approximation argument, we know
 $\Sigma$ is stable if and only if for every $\phi \in W^{1,2}_2(\Sigma)$ with compact support, we have
\begin{equation}
	\sum_{\Lambda \in \mathcal{A} }^{}\int_{ \Lambda} \left|\nabla_\Lambda \phi_\Lambda\right|^2-\left|A_\Lambda\right|^2\phi_\Lambda^2d\mu_\Lambda-\sum_{L\in \mathcal{L} }^{}
			\sum_{i =1}^{3}\int_{ L} ( \phi^i_L )^2\boldsymbol{H}_L\cdot \tau^i_Ld\mu_L\ge 0.
		\label{eq_stab}
\end{equation}

\section{Proof of Main Theorem}%
\label{sec:main_theoreem}

%We summarize the angle condition in the following form.
%
%\begin{definition}
%	\label{def_equil_angle}
%	We say $\Sigma=(\theta^1\Sigma^1,\cdots ,\theta^q \Sigma^q)$ has \textit{equilibrium angles} along $\Gamma$ if for each $1\le i,j\le q$, we have $\measuredangle (\nu^i,\nu^j)$ is a constant along $\Gamma$.
%	Here, we write $\measuredangle (\nu^i,\nu^j)$ as the directed angle between $\nu^i$ and $\nu^j$ when we fix an orientation along $\Gamma$.
%\end{definition}
%
%\begin{remark}
%	This is slightly different from the definition that appeared in \cite[Section 6]{Wang2022curvature}.
%	Indeed, we should use $\nu^i$ to define the concept of equilibrium angles instead of using $\tau^i$ since compatibility is related to $\nu^i$.
%	The angle condition that appeared in Theorem 1 of \cite{Wang2022curvature} should also be understood as the equilibrium angles define in Definition \ref{def_equil_angle}.
%\end{remark}

%With the above definitions, we can state our main theorem precisely.
%\begin{theorem}
%	\label{thm_bern}
%	Let $\Sigma=(\theta^1 \Sigma^1,\cdots ,\theta^q \Sigma^q)$ be a minimal multiple junction surface in $\mathbb{R}^3 $. We suppose $\Sigma$ is complete, stable and has quadratic area growth and equilibrium angles along $\Gamma$.
%	Then each $\Sigma^i$ is flat.
%\end{theorem}

From now on, we suppose $\Sigma$ is a complete minimal Plateau surface in $\mathbb{R}^3$.
We use $\left\{ \Omega^j \right\}_{j=1}^N$ to denote the set of all connected components of $\mathbb{R}^3\backslash \Sigma$. That is, we write
\[
	\mathbb{R}^3 \backslash \Sigma=\bigcup _{j=1}^N \Omega^j.
\]

Note that we allow $N=+\infty$.
For each connected component $\Omega^j$, we know it induces a locally constant function $f^j:=f^{\Omega^j}$. With these notations, let us prove our main theorem.

\begin{proof}
	[Proof of Theorem \ref{thm_main_Plat}]

%	Since $\Sigma$ has equilibrium angles along $\Gamma$, we can choose a unit normal vector field $W_1$ along $\Gamma$ such that $\measuredangle (W_1, \nu^i)$ is a constant function along $\Gamma$. By perturbing $W_1$, we can assume $\measuredangle (W_1, \nu^i)\in (0,\frac{\pi}{2})\cup (\frac{\pi}{2},\pi)\cup (\pi,\frac{3\pi}{2})\cup (\frac{3\pi}{2},2\pi)$.
%We choose another unit normal vector field $W_2$ along $\Gamma$ such that $\measuredangle (W_1,W_2)=\frac{\pi}{2}$.
%Clearly, we know $\measuredangle (W_2,\nu^i)=\measuredangle (W_1,\nu^i)-\frac{\pi}{2}$.

%If we define $c_1,c_2 \in C^{\infty}(\Sigma)$ by choosing $c_j^i=W_j\cdot \nu^i$ as a constant function on $\Sigma^i$, we can find $c_1,c_2 \in C^{\infty}_1(\Sigma)$.
%The idea is, we want to put $c_1,c_2$ into stability inequality (\ref{eq_stab}).
%Of course we need cut-off functions to help us since $c_1,c_2$ do not have compact support.

Fix a point $p_0 \in \Sigma$ and a positive integer $n \in \mathbb{N}$.
We define a logarithmic cutoff function $\zeta:\mathbb{R}^3 \rightarrow \mathbb{R} $ as
\[
	\zeta(p)=
	\begin{cases}
	1, & p \in B_1(p),\\
	1-\frac{\log \left|p-p_0\right|}{n}, & p \in B_{e^n}(p_0)\backslash B_{1}(p_0),\\
	0, & \text{otherwise}.
	\end{cases}
\]

Note that the function $f^i\zeta:=(f_\Lambda \zeta|_{\Lambda})_{\Lambda \in \mathcal{A}}$ satisfies the compatible condition. Then we can choose $\phi=f^i \zeta$ in our stability inequality (\ref{eq_stab}) to get
\begin{align}
	&\sum_{\substack{\Lambda \in \mathcal{A} \text{ and } \Lambda \in \overline{\Omega^j}} }^{}\int_{ \Lambda\cap B_1(p_0)}\left|A_\Lambda\right|^2 d\mu_\Lambda + 
	\sum_{L \in \mathcal{L}, L \in \overline{\Omega^j}}^{}\sum_{\Lambda^i_L \subset \overline{\Omega^j} }^{}\int_{L}  \zeta^2 \boldsymbol{H}_L \cdot \tau^i_L d\mu_L\nonumber \\
	\le{}& \sum_{\Lambda \in \mathcal{A} \text{ and }\Lambda \in \overline{\Omega^j} }^{}\int_{ \Lambda \cap (B_{e^{n}}(p_0)\backslash B_1(p_0))} \frac{1}{n^2\rho^2}d\mu_\Lambda.
	\label{eq_pf_stab}
\end{align}

Now we can take sum over all $j$ for the inequality (\ref{eq_pf_stab}) to get
\begin{align*}
	{} & 2\sum_{\Lambda \in \mathcal{A} }^{}\int_{ \Lambda \cap B_1(p_0)} \left|A_\Lambda\right|^2d\mu_\Lambda + 2 \sum_{L \in \mathcal{L} }^{}\sum_{i =1}^{3}\int_{ L} \zeta^2 \boldsymbol{H}_L \cdot \tau^i_L d\mu_L  \\
	\le{} & 2 \sum_{\Lambda \in \mathcal{A} }^{}\int_{ \Lambda \cap (B_{e^n}(p_0)\backslash B_1(p_0))}\frac{1}{n^2\rho^2}d\mu_\Lambda.
\end{align*}

The stationary condition implies $\sum_{i =1}^{3}\tau^i_L=0$ along $L$. Hence, we get
\begin{align*}
	 {} & \sum_{\Lambda \in \mathcal{A} }^{}\int_{ \Lambda \cap B_1(p_0)} \left|A_\Lambda\right|^2d\mu_\Lambda\le \sum_{k=1 }^{n}\int_{\Sigma\cap (B_{e^{k}}(p_0)\backslash B_{e^{k-1}}(p_0))} \frac{1}{n^2\rho^2}d\mu_\Sigma\\
	\le{} & \sum_{k=1 }^{n}\int_{\Sigma\cap (B_{e^{k}}(p_0)\backslash B_{e^{k-1}}(p_0))} \frac{1}{n^2e^{2k-2}}d\mu_\Sigma\le \sum_{k =1}^{n}\frac{Ce^2}{n^2}=\frac{C}{n}.
\end{align*}
Here, we have used the condition of quadratic area growth.
Let $n\rightarrow +\infty$, we know $A_\Lambda$ vanishes in $B_1(p_0)$.
By the arbitrariness of $p_0$, we know $\left|A _\Lambda\right|$ should vanish everywhere on $\Lambda$ for each $\Lambda \in \mathcal{A}$.
This shows $\Lambda$ is flat for each $\Lambda \in \mathcal{A}$.
\end{proof}

Now, we can give the proof of Corollary \ref{cor_main_plat}.

%\begin{corollary}
%	\label{cor_structure}
%	If $\Sigma$ is a multiple junction surface satisfying the conditions in Theorem \ref{thm_bern}, then $\Sigma$ can be written as $\Sigma=N\times \mathbb{R} $ after rigid motions in $\mathbb{R}^3$ where $N$ is an immersed network in $\mathbb{R}^2 $.
%\end{corollary}

\begin{proof}[Proof of Corollary \ref{cor_main_plat}]
	By Theorem \ref{thm_main_Plat}, we know each $\Lambda \in \mathcal{A}$ is a smooth domain in some plane and each $L \in \mathcal{L}$ is a line segment, ray or a straight line in $\mathbb{R}^3$.
	This implies $\Lambda$ should be congruent to a generalized polygon in $\mathbb{R}^2$ (could be unbounded). Note that the interior angle of $\Lambda$ is $\cos(-\frac{1}{3})$. So $\Lambda$ can only be unbounded and it could have at most 3 sides.
	Hence, we know $\Lambda$ is congruent to one of the following domains in $\mathbb{R}^2 $.
	\begin{itemize}
		\item A half-plane $\left\{(x_1,x_2)\in \mathbb{R}^2 : x_1> 0 \right\}$.
		\item A strip $\left\{ (x_1,x_2)\in \mathbb{R}^2 :a<x_1<b \right\}$ for some $a<b$.
		\item An angular region $\left\{ (x_1,x_2)\in \mathbb{R}^2 :\sqrt{2}x_2>\left|x_1\right| \right\}$.
		\item A 3-sided region given by $\left\{ (x_1,x_2)\in \mathbb{R}^2 :x_2+a>2\sqrt{2}\left|x_1\right|,x_2>0 \right\}$ for some $a>0$.
	\end{itemize}

	The first thing we may notice, if $\Sigma$ contains a plane, then it should be the union of disjoint parallel planes.
	Otherwise, we know $\Sigma$ is always connected by the maximum principle.

	If there exists some $L \in \mathcal{L}$ which is a straight line in $\mathbb{R}^3 $, then by induction, we can show that $\Lambda$ is either a half-plane or a strip for any $\Lambda \in \mathcal{A}$.
	Hence we know $\Sigma$ is isometric to $N\times \mathbb{R} $ where $N$ is a stationary network in $\mathbb{R}^2$.
	This is the case of $\Sigma_T=\emptyset $.

	Let us assume none of $L \in \mathcal{L}$ is a straight line. 
	Of course, we know $\Sigma_T\neq \emptyset$.
	If $\Sigma_T$ only contains one point and write $\Sigma_T=\left\{ p \right\}$, then every $L \in \mathcal{L}$ is a ray with $p$ as its endpoint.
	So $\Sigma$ is isometric to $T$.

	Similarly, if $\Sigma_T$ contains two points, we write $\Sigma_T=\left\{ p_1,p_2 \right\}$. Then we know the segment $\overline{p_1p_2}\in \mathcal{L}$ since $\Sigma$ is connected.
	Here, we denote $\overline{pq}$ the open line segment with endpoints $p,q$. Any element in $\mathcal{L}$ besides $\overline{p_1p_2}$ is a ray in $\mathbb{R}^3 $. Hence, the only possible $\Sigma$ is the one shown in Figure \ref{fig:flat-tplanes}.

	If $\Sigma_T$ contains more than two points. Suppose $\Sigma_T=\left\{ p_1,p_2,p_3,\cdots  \right\}$. Without loss of generality, we suppose $\overline{p_1p_2},\overline{p_2p_3}\in \mathcal{L}$.
	Let $\Lambda \in \mathcal{A}$ be the one such that $\overline{p_1p_2}\cup \overline{p_2p_3}\subset \overline{\Lambda}$.
	But this is impossible since $\Lambda$ can only have at most one side which is bounded.
	Hence, $\Sigma_T$ cannot contain more than two points.
\end{proof}

\subsection*{Acknowledge}%
\label{sub:acknowledge}

The author wishes to thank Prof. Jacob Bernstein for providing the extension of my work and some constructive comments and thank Prof. Francesco Maggi for his very detailed suggestion.
He also wishes to thank his advisor Prof. Martin Li for his constant support and encouragement.

\appendix
\section{Appendix}%
\label{sec:appendix}

\subsection{Extension of the results in \cite{Wang2022curvature}}%
\label{sub:extension_of_the_resulst_in_wang2022curvature}

We use the notations from \cite{Wang2022curvature}. We summarize those basic definitions here.

\begin{definition}
	\label{def_minimal_multiple}
	We say $\Sigma=(\theta^1\Sigma^1,\cdots ,\theta^q \Sigma^q)$ is a \textit{minimal multiple junction surface} if the following conditions hold.
	\begin{itemize}
		\item Each $\Sigma^i$ is an immersed (connected, oritentable) minimal surface with smooth boundary $\partial \Sigma^i$ and $\partial \Sigma^i=\partial \Sigma^j$ for $1\le i,j\le q$.
			We write $\Gamma=\partial \Sigma^i$.
		\item $\sum_{i =1}^{q}\theta^i \tau^i=0$ along $\Gamma$ where $\tau^i$ is the outer normal vector field $\Gamma$ in $\Sigma^i$.
	\end{itemize}
\end{definition}

We use $\nu^i$ to denote the unit normal vector field on $\Sigma^i$.

Now let us define a distance function $d(\cdot,\cdot)$ on $\Sigma$.
Let $d^i(\cdot,\cdot)$ be the intrinsic distance function on $\Sigma^i$ ($d^i(p_1,p_2)=+\infty$ if $p_1,p_2$ belongs to the different components of $\Sigma^i$). 
We can define
\begin{align*}
	d(x,y):={} & \inf \{ \sum_{k =0}^{l-1}
	d ^{i_k}(x_k,x_{k+1}):x_0=x,x_{l}=y, x_1,\cdots ,x_{l-1} \in \Gamma,
\\
	 {} & i_0=i, i_{l-1}=j, 1\le i_1,\cdots ,i_{l-2}\le q,\quad \text{ for }l \in \mathbb{N}\}.
\end{align*}

We denote $B_r^\Sigma(x)=\left\{ y \in \Sigma: d(x,y) <r \right\}$ the intrinsic ball on $\Sigma$.

\begin{definition}
	\label{def_complete}
	We say a minimal multiple junction surface $\Sigma$ is complete if it is complete with respect to the distance function $d(\cdot,\cdot)$.
\end{definition}

\begin{definition}
	\label{def_equil_angle}
	We say $\Sigma=(\theta^1\Sigma^1,\cdots ,\theta^q \Sigma^q)$ has \textit{equilibrium angles} along $\Gamma$ if for each $1\le i,j\le q$, we have $\measuredangle (\nu^i,\nu^j)$ is a constant along $\Gamma$.
	Here, we write $\measuredangle (\nu^i,\nu^j)$ as the directed angle between $\nu^i$ and $\nu^j$ when we fix an orientation along $\Gamma$.
\end{definition}

\begin{remark}
	This is slightly different from the definition that appeared in \cite[Section 6]{Wang2022curvature}.
	Indeed, we should use $\nu^i$ to define the concept of equilibrium angles instead of using $\tau^i$ since compatibility is related to $\nu^i$.
	The angle condition that appeared in Theorem 1 of \cite{Wang2022curvature} should also be understood as the equilibrium angles define in Definition \ref{def_equil_angle}.
\end{remark}

\begin{theorem}
	(cf. \cite[Theorem 4]{Wang2022curvature})
	Let $\Sigma$ be a complete stable minimal multiple junction surface in $\mathbb{R}^3 $.
	Then, for any $\phi =(\phi^i)_{i=1}^q$ with $\phi^i \in W^{1,2}(\Sigma^i)$ satisfying the compatible condition that $\phi^i=W\cdot \nu^i$ $\mathcal{H}^{1}$-a.e., on $\Gamma$ for some $L^2$ vector field $W$ along $\Gamma$, we have
	\[
		0\le 
		\sum_{i =1}^{q}
		\int_{ \Sigma^i} \left( \left|\nabla _{\Sigma^i}\phi^i\right|^2-\left|A _{\Sigma^i}\right|^2\left(\phi^i\right)^2\right)\theta^i d\Sigma^i-\int_{ \Gamma} (\phi^i)^2 \boldsymbol{H}_\Gamma \cdot \tau^i \theta^i d\Gamma.
	\]
	where $\boldsymbol{H}_\Gamma$ is the geodesic curvature vector of $\Gamma$.
	\label{thm_stab_ineq_multi}
\end{theorem}

With the above definitions, we can state our main theorem.
\begin{theorem}
	\label{thm_bern}
	Let $\Sigma=(\theta^1 \Sigma^1,\cdots ,\theta^q \Sigma^q)$ be a minimal multiple junction surface in $\mathbb{R}^3 $. We suppose $\Sigma$ is complete, stable and has quadratic area growth and equilibrium angles along $\Gamma$.
	Then each $\Sigma^i$ is flat.
\end{theorem}

\begin{proof}
	Since $\Sigma$ has equilibrium angles along $\Gamma$, we can choose a smooth unit normal vector field $W_1$ along $\Gamma$ such that $\measuredangle (W_1, \nu^i)$ is a constant function along $\Gamma$. By perturbing $W_1$, we can assume $\measuredangle (W_1, \nu^i)\in (0,\frac{\pi}{2})\cup (\frac{\pi}{2},\pi)\cup (\pi,\frac{3\pi}{2})\cup (\frac{3\pi}{2},2\pi)$.
We choose another unit normal vector field $W_2$ along $\Gamma$ such that $\measuredangle (W_1,W_2)=\frac{\pi}{2}$.
Clearly, we know $\measuredangle (W_2,\nu^i)=\measuredangle (W_1,\nu^i)-\frac{\pi}{2}$.

If we define $c_1,c_2 \in C^{\infty}(\Sigma)$ by choosing $c_j^i=W_j\cdot \nu^i$ as a constant function on $\Sigma^i$, we can find $c_1,c_2 \in C^{\infty}_1(\Sigma)$.
The idea is, we want to put $c_1,c_2$ into stability inequality (\ref{eq_stab}).
Of course we need cut-off functions to help us since $c_1,c_2$ do not have compact support.

Fix a point $p_0 \in \Sigma$ and a positive integer $n \in \mathbb{N}$, we denote $\rho(p):=\mathrm{dist}(p,p_0)$.
Note that we allow $p_0\notin \Gamma$.
We choose test functions $\phi_1,\phi_2$ defined as
\[
	\phi_j^i=
	\begin{cases}
	c^i_j(p), & p \in B_1^\Sigma(p_0)\cap \Sigma^i,\\
	c^i_j(p)\left( 1-\frac{\log \rho(p)}{n} \right),& p \in (B^\Sigma _{e^n}(p_0)\backslash B_1^\Sigma(p_0))\cap \Sigma^i, \\
	0, & \text{otherwise}.
	\end{cases}
\]

%Note that $\phi_1,\phi_2$ are not smooth.
Note that $\phi_1,\phi_2$ satisfy the compatible condition in the sense that
\[
	\phi_j^i=\tilde{W}_j \cdot \nu^i,\quad \text{ along }\partial \Sigma^i
\]
for some Lipschitz normal vector field $\tilde{W}_j$ defined along $\Gamma$.
In particular, we can write $\tilde{W}_j$ explicitly as,
\[
	\tilde{W}_j=
	\begin{cases}
	W_j(p), & p \in B_1^\Sigma(p_0)\cap \Gamma\\
	W_j(p)\left( 1-\frac{\log \rho(p)}{n} \right), & p \in (B_{e^n}^\Sigma(p_0)\backslash B_1^\Sigma(p_0))\cap \Gamma,\\
	0,& \text{otherwise}.
	\end{cases}
\]

%Hence, we can use smooth function in $C^{\infty}_1(\Sigma)$ to approach $\phi_1,\phi_2$.
%The standard approximation argument shows, the stability inequality \eqref{eq_stab} is valid if we choose $\phi=\phi_1$ or $\phi=\phi_2$. (See \cite{Wang2022curvature} for the discussion of Sobolev function spaces with compatible conditions.)

Hence, we can choose $\phi=\phi_j$ in Theorem \ref{thm_stab_ineq_multi} and get the following inequality.
\begin{equation}
	\sum_{i =1}^{q}\int_{ \Gamma} (\phi^i_j)^2 \boldsymbol{H}_\Gamma \cdot \tau^i \theta^i d\Gamma+\int_{ \Sigma^i}\left|A _{\Sigma^i}\right|^2 (\phi^i_j)^2 \theta^i d\Sigma^i \le
	\sum_{i =1}^{q}\int_{ \Sigma^i} \left|\nabla _{\Sigma^i}\phi^i_j\right|^2 \theta^i d\Sigma^i.
	\label{eq_pf_stab}
\end{equation}

Note that we have
\begin{align*}
	(\phi_1^i)^2+(\phi_2^i)^2={} & (\tilde{W}_1 \cdot \nu^i)^2+(\tilde{W}_2\cdot \nu^i)^2=|\tilde{W}_1|^2,
\end{align*}
since $\measuredangle (\tilde{W}_1,\tilde{W}_2)=\frac{\pi}{2}$ and $|\tilde{W}_1|=|\tilde{W}_2|$ if $|\tilde{W}_1|\neq 0$.
Therefore,
\begin{align*}
	\sum_{i =1}^{q}\left[ (\phi_1^i)^2+(\phi_2^i)^2 \right] \boldsymbol{H}_\Gamma \cdot\tau^i \theta^i={} & |\tilde{W}_1|^2\sum_{i =1}^{q} \boldsymbol{H}_\Gamma \cdot (\theta^i \tau^i)=0
\end{align*}
by the definition of minimal multiple junction surfaces.
Combining \eqref{eq_pf_stab}, we have
\begin{align*}
	&\sum_{i =1}^{q}\int_{ B_1^\Sigma(p_0)} \left[ (c_1^i)^2+(c_2^i)^2 \right] \left|A _{\Sigma^i}\right|^2\theta^i d\Sigma^i\\
	\le{}&
	\sum_{i =1}^{q}\int_{\Sigma^i} 
	\left[ \left|\nabla _{\Sigma^i}\phi_1^i\right|^2+
	\left|\nabla _{\Sigma^i}\phi_2^i\right|^2 \right]\theta^i d\Sigma^i\\
	={}& \sum_{i =1}^{q}\left[ (c_1^i)^2+(c_2^i)^2 \right] \int_{ \Sigma^i \cap (B^\Sigma _{e^n}(p_0)\backslash B^\Sigma _{1}(p_0))} \frac{|\nabla_{\Sigma^i}\rho|^2}{n^2 \rho^2}\theta^i d\Sigma^i \\
	\le{}&  \sum_{i =1}^{q}\left[ (c_1^i)^2+(c_2^i)^2 \right] \sum_{k =1}^{n}
	\int_{ \Sigma^i \cap (B^\Sigma _{e^k}(p_0)\backslash B^\Sigma _{e^{k-1}}(p_0))} 
	\frac{1}{n^2 e^{2k-2}}\theta^i d\Sigma^i\\
	\le{}& \sum_{i =1}^{q}\left[ (c_1^i)^2+(c_2^i)^2 \right] \sum_{k =1}^{n} \frac{Ce^2\theta^i}{n^2} = \frac{C}{n}\sum_{i =1}^{q}\left[ (c_1^i)^2+(c_2^i)^2 \right]\theta^i
\end{align*}

Here, we have used the quadratic area growth condition $\left|\Sigma^i \cap B^\Sigma_r(p_0)\right|\le Cr^2$.
Note that $(c_1^i)^2+(c_2^i)^2>0$ for each $i$, by choosing $n$ large enough, we can conclude $\left|A _{\Sigma^i}\right|$ vanishes in $B_1^\Sigma(p_0)\cap \Sigma^i$ for each $1\le i\le q$.
By the arbitrariness of $p_0$, we know $\left|A _{\Sigma^i}\right|$ should vanish everywhere on $\Sigma^i$.

This shows $\Sigma^i$ is flat for each $1\le i\le q$.
\end{proof}

%Now, we can give the proof of Corollary \ref{cor_structure}.

%\begin{corollary}
%	\label{cor_structure}
%	If $\Sigma$ is a multiple junction surface satisfying the conditions in Theorem \ref{thm_bern}, then $\Sigma$ can be written as $\Sigma=N\times \mathbb{R} $ after rigid motions in $\mathbb{R}^3$ where $N$ is an immersed network in $\mathbb{R}^2 $.
%\end{corollary}

%\begin{proof}[Proof of Corollary \ref{cor_structure}]
%	By Theorem \ref{thm_bern}, we know each component of $\Sigma^i$ is a smooth domain in some planes and each component of $\Gamma$ is the subset of the intersection of two different planes.
%	Since $\Sigma$ is complete, we know the component of $\Gamma$ should be a straight line.
%	
%	Hence, let $U$ be one of the components of $\Sigma^i$, then we know either $\partial U$ is a straight line, or $\partial U$ is a union of two parallel lines.
%	So $U$ is either a half-plane or a strip in a plane.
%	Since $\Sigma$ is connected, we can find the straight lines contained in $\Sigma$ are parallel to each other.
%	Hence, after a rigid motion, we know $\Gamma=P\times \mathbb{R} $ where $P=\left\{ p_1,p_2,\cdots  \right\}\subset \mathbb{R}^2 $ is a collection of points (may happen $p_i=p_j$ for $i\neq j$ since $\Sigma$ is immersed).
%	This will make sure the component of $\Sigma^i$ should have the form $S\times \mathbb{R} $ where $S$ is a segment or a ray with its endpoints in $P$.
%
%	Therefore, we can write $\Sigma=N\times \mathbb{R} $.
%	Note that we can find $N$ is also a stationary network in $\mathbb{R}^3 $ if we impose suitable densities by projection.
%\end{proof}

\bibliographystyle{alpha}
\bibliography{../references}

\end{document}